\documentclass{amsart}
\usepackage{latexsym}
\usepackage{amstext}
\usepackage{amsmath}
\usepackage{amssymb}
\usepackage{amsthm}
\usepackage{url}
\usepackage{tikz-cd}

\theoremstyle{plain}
\newtheorem{theorem}{Theorem}[section]

\newtheorem{proposition}[theorem]{Proposition}

\newtheorem{corollary}[theorem]{Corollary}
\newtheorem{def-thm}[theorem]{Definition-Theorem}
\newtheorem{lemma}[theorem]{Lemma}

\theoremstyle{definition}
\newtheorem{definition}[theorem]{Definition}

\newtheorem{remark}[theorem]{Remark}

\def \O {\mathcal{O}}
\def \II {\mathcal{I}}
\def \SS {\mathcal{S}}
\def \FF {\mathcal{F}}
\def \LL {\mathcal{L}}

\def \BB {\mathcal{B}}
\def \bt {\mathbf{t}}

\DeclareMathOperator{\dv}{div}

\DeclareMathOperator{\codim}{codim}
\DeclareMathOperator{\Proj}{Proj}

\allowdisplaybreaks
\begin{document}

\title[A generalized Schmidt subspace theorem]{A generalized Schmidt subspace theorem for closed subschemes}

\author{Gordon Heier}\author{Aaron Levin}
\address{Department of Mathematics\\University of Houston\\4800 Calhoun Road\\ Houston, TX 77204\\USA}
\email{heier@math.uh.edu}

\address{Department of Mathematics\\Michigan State University\\619 Red Cedar Road\\ East Lansing, MI 48824\\USA}
\email{adlevin@math.msu.edu}
\thanks{The second author was supported in part by NSF grant DMS-1102563.}



\begin{abstract}
We prove a generalized version of Schmidt's subspace theorem for closed subschemes in general position in terms of suitably defined Seshadri constants with respect to a fixed ample divisor. Our proof builds on previous work of Evertse and Ferretti, Corvaja and Zannier, and others, and uses standard techniques from algebraic geometry such as notions of positivity, blowing-ups and direct image sheaves. As an application, we recover a higher-dimensional Diophantine approximation theorem of K.~F.~Roth-type due to D.~McKinnon and M.~Roth with a significantly shortened proof, while simultaneously extending the scope of the use of Seshadri constants in this context in a natural way.
\end{abstract}

\maketitle

\section{Introduction}
In the theory of higher-dimensional Diophantine approximation, Schmidt's subspace theorem has inspired a long series of generalizations by many authors. Among these generalizations is the following theorem due to Evertse and Ferretti \cite{ef_festschrift} (see also the related results of Corvaja and Zannier \cite{CZ}).

\begin{theorem}[Evertse-Ferretti]
\label{EF}
Let $X$ be a projective variety of dimension $n$ defined over a number field $k$.  Let $S$ be a finite set of places of $k$.  For each $v\in S$, let $D_{0,v},\ldots, D_{n,v}$ be effective Cartier divisors on $X$, defined over $k$, in general position.  Suppose that there exists an ample Cartier divisor $A$ on $X$ and positive integers $d_{i,v}$ such that $D_{i,v}\sim d_{i,v}A$ for all $i$ and for all $v\in S$.  Let $\epsilon>0$.  Then there exists a proper Zariski-closed subset $Z\subset X$ such that for all points $P\in X(k)\setminus Z$,
\begin{equation}
\label{EFeq}
\sum_{v\in S}\sum_{i=0}^n \frac{\lambda_{D_{i,v},v}(P)}{d_{i,v}}< (n+1+\epsilon)h_A(P).
\end{equation}
\end{theorem}

Here, $\lambda_{D_{i,v},v}$ is a local height function (also known as a Weil function) associated to the divisor $D_{i,v}$ and place $v$ in $S$, and $h_A$ is a global (absolute) height associated to $A$.  We will also use throughout the theory of heights and local heights associated to closed subschemes (see \cite{silverman_87}).

A key condition in the above theorem is the required linear equivalence $D_{i,v}\sim d_{i,v}A$ for all $i$ and for all $v\in S$. It is very natural to seek versions of this theorem in which this condition is suitably weakened. In this spirit, the 
second author in \cite{levin_duke} proved the following generalization of Evertse and Ferretti's theorem to ample divisors which are numerically equivalent up to a constant.

\begin{theorem}[Levin]
\label{EFnum}
Let $X$ be a projective variety of dimension $n$ defined over a number field $k$.  Let $S$ be a finite set of places of $k$.  For each $v\in S$, let $D_{0,v},\ldots, D_{n,v}$ be effective Cartier divisors on $X$, defined over $k$, in general position.  Suppose that there exists an ample Cartier divisor $A$ on $X$ and positive integers $d_{i,v}$ such that $D_{i,v}\equiv d_{i,v}A$ for all $i$ and for all $v\in S$.  Let $\epsilon>0$.  Then there exists a proper Zariski-closed subset $Z\subset X$ such that for all points $P\in X(k)\setminus Z$,
\begin{equation*}
\sum_{v\in S}\sum_{i=0}^n \frac{\lambda_{D_{i,v},v}(P)}{d_{i,v}}< (n+1+\epsilon)h_A(P).
\end{equation*}
\end{theorem}
In order to further relax the assumptions of the above theorem, we incorporate the concept of Seshadri constants to measure positivity. Moreover, invoking Seshadri constants naturally allows for the divisors $D_{i,v}$ to be replaced with closed subschemes of arbitrary dimension. We obtain the following theorem generalizing the Schmidt subspace theorem, which is the main theorem of this note.

\begin{theorem}
\label{mthm}
Let $X$ be a projective variety of dimension $n$ defined over a number field $k$.  Let $S$ be a finite set of places of $k$.  For each $v\in S$, let $Y_{0,v},\ldots, Y_{n,v}$ be closed subschemes of $X$, defined over $k$, and in general position.  Let $A$ be an ample Cartier divisor on $X$, and $\epsilon>0$.  Then there exists a proper Zariski-closed subset $Z\subset X$ such that for all points $P\in X(k)\setminus Z$,
\begin{equation*}
\sum_{v\in S}\sum_{i=0}^n \epsilon_{Y_{i,v}}(A) \lambda_{Y_{i,v},v}(P)< (n+1+\epsilon)h_A(P).
\end{equation*}
\end{theorem}

We refer to Definition \ref{Sesh} for the definition of the Seshadri constants $\epsilon_{Y_{i,v}}(A)$ and Definition \ref{gen_pos_def} for the notion of general position used here.

We record as an immediate corollary the following simplified version in the case that the closed subschemes are Cartier divisors. This version stems from an earlier unpublished manuscript, and we presented it at a 2014 meeting in Banff on Vojta's Conjectures. Note that if $D$ is an effective Cartier divisor (which we also view as a closed subscheme), then $\tilde{X}=X$ in Definition \ref{Sesh} and the Seshadri constants simplify to
\begin{align*}
\epsilon_D(A)=\sup\{\gamma\in {\mathbb{Q}}^{\geq 0}\mid A-\gamma D\text{ is $\mathbb{Q}$-nef}\}.
\end{align*}

\begin{corollary}
\label{mthm_codim_1}
Let $X$ be a projective variety of dimension $n$ defined over a number field $k$.  Let $S$ be a finite set of places of $k$.  For each $v\in S$, let $D_{0,v},\ldots, D_{n,v}$ be effective Cartier divisors on $X$, defined over $k$, in general position.  Let $A$ be an ample Cartier divisor on $X$ and let $c_{i,v}$ be rational numbers such that $A-c_{i,v}D_{i,v}$ is a nef $\mathbb{Q}$-divisor for all $i$ and for all $v\in S$.  Let $\epsilon>0$.  Then there exists a proper Zariski-closed subset $Z\subset X$ such that for all points $P\in X(k)\setminus Z$,
\begin{equation*}
\sum_{v\in S}\sum_{i=0}^n c_{i,v}\lambda_{D_{i,v},v}(P)< (n+1+\epsilon)h_A(P).
\end{equation*}
\end{corollary}

We note also that the above corollary yields Theorem \ref{EFnum} as a special case as, under its hypotheses, we plainly can take $c_{i,v}=\frac{1}{d_{i,v}}$.\par

By a standard argument, Theorem \ref{mthm} yields a similar inequality for proximity functions $m_{Y,S}(P)=\sum_{v\in S} \lambda_{Y,v}(P)$ and an arbitrary number of closed subschemes in general position:

\begin{corollary}
\label{EFnum_prox}
Let $X$ be a projective variety of dimension $n$ defined over a number field $k$.  Let $S$ be a finite set of places of $k$.  Let $Y_0,\ldots, Y_q$ be closed subschemes of $X$, defined over $k$, and in general position.  Let $A$ be an ample Cartier divisor on $X$, and $\epsilon>0$.  Then there exists a proper Zariski-closed subset $Z\subset X$ such that for all points $P\in X(k)\setminus Z$,
\begin{equation*}
\sum_{i=0}^q \epsilon_{Y_{i}}(A) m_{Y_{i},S}(P)< (n+1+\epsilon)h_A(P).
\end{equation*}
\end{corollary}
For the next corollary, we observe that according to Definition \ref{gen_pos_def}, for a closed subscheme $Y$ of codimension $r$, the elements of the list obtained by repeating $Y$ up to $r$ times are in general position.

\begin{corollary}\label{repeated_cor}
Let $X$ be a projective variety of dimension $n$ defined over a number field $k$.  Let $S$ be a finite set of places of $k$.  For each $v\in S$, let $Y_v$ be a closed subscheme of $X$, defined over $k$, of codimension $\codim Y_v$ in $X$.  Let $A$ be an ample Cartier divisor on $X$, and $\epsilon>0$. Then there exists a proper Zariski-closed subset $Z\subset X$ such that for all points $P\in X(k)\setminus Z$,
\begin{equation*}
\sum_{v\in S}(\codim Y_v)\epsilon_{Y_v}(A)\lambda_{Y_{v},v}(P)< (n+1+\epsilon)h_A(P).
\end{equation*}
\end{corollary}

In particular, if $Y_v$ is equal to a fixed point $x$ for all $v\in S$, then Corollary \ref{repeated_cor} yields the existence of a proper Zariski-closed subset $Z\subset X$ such that
\begin{equation*}
\sum_{v\in S}\epsilon_x(A)\lambda_{x,v}(P)< \left(\frac{n+1}{n}+\epsilon\right)h_A(P)
\end{equation*}
for all $P\in X(k)\setminus Z$. This statement appears as Theorem 6.2 (alternative statement) in \cite{McK_R}.  More generally, if $\codim Y_v=r$ for all $v\in S$, the inequality we obtain in Corollary \ref{repeated_cor} takes the form
\begin{equation*}
\sum_{v\in S}\epsilon_{Y_v}(A)\lambda_{Y_{v},v}(P)< \left(\frac{n+1}{r}+\epsilon\right)h_A(P).
\end{equation*}\par

Recently, Ru and Wang \cite{RW} (see also \cite{Grievea, Grieveb} for versions over function fields) have given a different generalization of McKinnon and Roth's results \cite{McK_R} to arbitrary closed subschemes:

\begin{theorem}[Ru-Wang]\label{theoremRW}
Let $X$ be a projective variety defined over a number field $k$.  Let $S$ be a finite set of places of $k$.  Let $Y_0,\ldots, Y_q$ be closed subschemes of $X$, defined over $k$, such that at most $\ell$ of the closed subschemes meet at any point of $X$.  Let $A$ be a big Cartier divisor on $X$ and let $\epsilon>0$.  Let
\begin{align*}
\beta_{A,Y_i}=\lim_{N\to\infty} \frac{\sum_{m=1}^\infty h^0(\tilde{X}_i,N\pi_i^*A-mE_i)}{Nh^0(X,NA)}, \quad i=0,\ldots, q,
\end{align*}
where $\pi_i:\tilde{X}_i\to X$ is the blowing-up of $X$ along $Y_i$, with associated exceptional divisor $E_i$.  Then there exists a proper Zariski-closed subset $Z\subset X$ such that for all points $P\in X(k)\setminus Z$,
\begin{equation*}
\sum_{i=0}^qm_{Y_{i},S}(P)< \ell\left(\max_{0\leq i\leq q}\left\{\beta_{A,Y_i}^{-1}\right\}+\epsilon\right)h_A(P).
\end{equation*}
\end{theorem}

When the closed subschemes $Y_i=y_i$ are distinct points in $X$ and $\dim X=n$, then one may take $\ell=1$ in Theorem \ref{theoremRW}, and McKinnon and Roth have shown that
\begin{align*}
\beta_{A,y_i}\geq \frac{n}{n+1}\epsilon_{y_i}(A),
\end{align*}
yielding a connection to Seshadri constants.  More generally, combining the method of proof of our main theorem with work of Autissier \cite{autissier}, we show in Theorem~\ref{shineq_thm} that if $X$ is non-singular and  $Y$ has codimension $r$ in $X$, then
\begin{align}
\label{Seshineq}
\beta_{A,Y}\geq \frac{r}{n+1}\epsilon_{Y}(A).
\end{align}

This yields an alternative proof of Corollary \ref{repeated_cor} through Ru-Wang's theorem in the non-singular case (more precisely, one must combine \eqref{Seshineq} with the proof of Ru-Wang's theorem in \cite{RW} to obtain the exact statement of Corollary \ref{repeated_cor} in the non-singular case).  The main theorem (Theorem \ref{mthm}) and its other corollaries (Corollary \ref{mthm_codim_1} and Corollary \ref{EFnum_prox}), however, do not follow from a na\"ive direct application of Ru-Wang's theorem, even in the non-singular case.  For instance, even in the simple case when the $Y_i$ are hypersurfaces in general position on $\mathbb{P}^n$, due to the factor $\ell$ on the right-hand side of the inequality, Ru-Wang's theorem doesn't recover Evertse-Ferretti's Theorem \ref{EF} (see \cite[Th.~1.7]{RW} and the discussion of Ru-Vojta's result \cite{RuVojta} following it).

The proof of Theorem \ref{mthm} relies on an application of Evertse and Ferretti's inequality \eqref{EFeq} to suitably constructed linearly equivalent divisors associated to the closed subschemes $Y_{i,v}$ and the ample divisor $A$ of Theorem \ref{mthm}.  In contrast to the proofs given in \cite{McK_R} and \cite{RW}, Seshadri constants appear directly in our arguments, and we do not make any use of inequalities of the form \eqref{Seshineq}.  In part due to this, our proofs are substantially simpler, at least under the assumption of Theorem \ref{EF}.

Using the well-known correspondence between statements in Diophantine approximation and Nevanlinna theory \cite{Vojta_LNM}, the proof of Theorem \ref{mthm} can be adapted to prove the following generalization of the Second Main Theorem in Nevanlinna theory:
\begin{theorem}
\label{mthmNev}
Let $X$ be a complex projective variety of dimension $n$. Let $Y_0,\ldots, Y_q$ be closed subschemes of $X$, $f:\mathbb{C}\to X$ a holomorphic map with Zariski dense image, $A$ an ample Cartier divisor on $X$, and $\epsilon>0$.  Then 
\begin{equation*}
\int_{0}^{2\pi}\max_J \sum_{j\in J} \epsilon_{Y_j}(A) \lambda_{Y_j}(f(re^{i\theta}))\frac{d\theta}{2\pi}\leq_{\operatorname{exc}} (n+1+\epsilon)T_{f,A}(r),
\end{equation*}
where the maximum is taken over all subsets $J$ of $\{0,\dots, q\}$ such that the closed subschemes $Y_j$,  $j\in J$, are in general position, and the notation $\leq_{\operatorname{exc}}$ means that the inequality holds for all $r\in (0,\infty)$ outside of a set of finite Lebesgue measure.
\end{theorem}

We refer to \cite{yamanoi_04} for the basic notation and development of Nevanlinna theory with respect to closed subschemes (i.e., an analogue of Silverman's theory \cite{silverman_87} in Nevanlinna theory).  Along with this theory, the proof of Theorem \ref{mthmNev} follows the proof of Theorem \ref{mthm}, but one must substitute the use of Evertse and Ferretti's result (Theorem \ref{EF}) with an appropriate application of the analogous result in Nevanlinna theory due to Ru \cite{ru_09}.  Ru's ``main result" in \cite{ru_09} isn't stated in a manner analogous to Theorem \ref{EF}; however, the proof in \cite{ru_09} proceeds via a reduction to a precise analogue of Theorem \ref{EF} (see Eq.~(3.2) and Eq.~(3.20)--(3.22) in \cite{ru_09}).  Analogues in Nevanlinna theory of Corollary \ref{mthm_codim_1}, Corollary \ref{EFnum_prox}, and Corollary \ref{repeated_cor} follow immediately.  As the proofs are reasonably straightforward (with the above substitutions), we omit the details.  In another direction, we will explore the implications of Theorem \ref{mthm} for the degeneracy of integral points in a subsequent paper.\par

The following Section \ref{background_material} contains background material from algebraic geometry and Diophantine approximation. The proof of the main theorem, and its application to Corollary \ref{repeated_cor}, will be given in Section \ref{mthm_proof}. In the final Section \ref{beta_ineq}, we establish inequality \eqref{Seshineq}, clarifying the relationship with the work of Ru-Wang.

\section{Background material}\label{background_material}
We begin by recalling some background material related to blowing-ups, based on \cite[Section II.7]{Hartshorne}. Note that throughout the paper, we will alternatingly use the language of invertible sheaves, line bundles, and divisors in order to properly align with statements and expressions from the literature and also to streamline our own notation. Let $X$ be a projective variety and $Y$ a closed subscheme of $X$, corresponding to a coherent sheaf of ideals $\II_Y$. Consider the sheaf of graded algebras $\SS=\bigoplus_{d\geq 0} \II_Y^d$, where $\II_Y^d$ is the $d$-th power of $\II_Y$, with the convention that $\II_Y^0=\O_X$. Then $\tilde X := \Proj \SS$ is called the {\it blowing-up of $X$ with respect to $\II_Y$}, or, alternatively, {\it the blowing-up of $X$ along $Y$}.\par
It should be noted that the corresponding morphism $\tilde X \to X$ in general does not necessarily contract a divisor on $\tilde X$. We will use the following well-known proposition.

\begin{proposition}[{\cite[Proposition II.7.13(a)]{Hartshorne}}]
Let $X$ be a projective variety, $\II_Y$ a coherent sheaf of ideals corresponding to a closed subscheme $Y$, and let $\pi:\tilde X\to X$ be the blowing-up with respect to $\II_Y$. Then the inverse image ideal sheaf $\tilde \II_Y = \pi^{-1}\II_Y \cdot \O_{\tilde X}$ is an invertible sheaf on $\tilde X$.
\end{proposition}
In fact, according to the proof of \cite[Proposition II.7.13(a)]{Hartshorne}, this invertible sheaf is $\O_{\tilde X}(1)$, as defined in \cite[p.~160, Construction]{Hartshorne}. Throughout the paper, we let $E$ denote an effective Cartier divisor in $\tilde X$ whose associated invertible sheaf is the dual of $\pi^{-1}\II_Y \cdot \O_{\tilde X}$. As we already remarked, the divisor $E$ might not be exceptional in the sense that it may not be contracted, but that turns out to be irrelevant to our argument. Moreover, \cite[Ex.~II.7.14(b)]{Hartshorne} gives us

\begin{lemma}
\label{amplelem}
Let $Y$ be a closed subscheme on a projective variety $X$.  Let $\pi:\tilde{X}\to X$ be the blowing-up along $Y$ and let $A$ be an ample Cartier divisor on $X$.  Then for all sufficiently small positive $\epsilon\in \mathbb{Q}$, the $\mathbb{Q}$-divisor $\pi^*A-\epsilon E$ is $\mathbb{Q}$-ample on $\tilde{X}$, where $E$ is an effective Cartier divisor on $\tilde X$ whose associated invertible sheaf is the dual of $\pi^{-1}\II_Y \cdot \O_{\tilde X}$.
\end{lemma}

We now define the appropriate notion of a {\it Seshadri constant} in this context. When $X$ is non-singular, this definition was made in Definition 1.1 and Remark 1.3 of \cite{CEL} by Cutkosky-Ein-Lazarsfeld (motivated in turn by earlier work of Paoletti \cite{Paoletti_JDG}). We expand their definition in a natural way to the singular case, which, while enabling us to obtain the desired Diophantine approximation results in full generality, may also be of independent interest.

\begin{definition}\label{Sesh}
Let $Y$ be a closed subscheme of a projective variety $X$ and let $\pi:\tilde{X}\to X$ be the blowing-up of $X$ along $Y$. Let $A$ be a nef Cartier divisor on $X$. We define the {\it Seshadri constant} $\epsilon_Y(A)$ of $Y$ with respect to $A$ to be the real number
\begin{align*}
\epsilon_Y(A)=\sup\{\gamma\in {\mathbb{Q}}^{\geq 0}\mid \pi^*A-\gamma E\text{ is $\mathbb{Q}$-nef}\},
\end{align*}
where $E$ is an effective Cartier divisor on $\tilde X$ whose associated invertible sheaf is the dual of $\pi^{-1}\II_Y \cdot \O_{\tilde X}$.
\end{definition}
For closed subschemes we work with the following notion of being in {\it general position}. As was remarked in the introduction, for a closed subscheme $Y$ of codimension $r$, the elements of the list obtained by repeating $Y$ up to $r$ times are, somewhat counterintuitively, in general position according to this definition.

\begin{definition}\label{gen_pos_def}
If $X$ is a projective variety of dimension $n$, we say that closed subschemes $Y_1,\ldots, Y_q$ of $X$ are in {\it general position} if for every subset $I\subset\{1,\ldots, q\}$ with $|I|\leq n+1$ we have $\codim \cap_{i\in I} Y_i\geq |I|$, where we use the convention that $\dim \emptyset=-1$. If $V$ is a subset of $X$, we say that closed subschemes $Y_1,\ldots, Y_q$ of $X$ are in {\it general position outside of $V$} if for every subset $I\subset\{1,\ldots, q\}$ with $|I|\leq n+1$ we have $\codim \left(\left(\cap_{i\in I} Y_i\right)\backslash V\right)\geq |I|$.\par
\end{definition}

\begin{remark}
When $Y_i=D_i$, $i=1,\ldots, q$, is an ample effective divisor for each $i$ (as in Theorem \ref{EF}), Definition \ref{gen_pos_def} is easily seen to be equivalent to other familiar notions of general position.  If $q>n$, the ample divisors $D_1,\ldots, D_q$ are in general position if and only if the intersection of any $n+1$ distinct divisors $D_i$ is empty.  In general, the ample divisors $D_1,\ldots, D_q$ are in general position if and only if for any subset $I\subset\{1,\ldots, q\}$ with $|I|\leq n+1$, we have an equality $\codim \cap_{i\in I} D_i=|I|$. 
\end{remark}

We end this section by recalling some of the basic properties of height functions.  Classically, one associates a height (or local height) to a (Cartier) divisor on a projective variety.  More generally, there is a theory of heights (and local heights) associated to arbitrary closed subschemes of a projective variety.  We give here a quick summary of the relevant properties of such heights and refer the reader to Silverman's paper \cite{silverman_87} for the general theory and details.

Let $Y$ be a closed subscheme of a projective variety $X$, both defined over a number field $k$.  For any place $v$ of $k$, one can associate a local height function (or Weil function) $\lambda_{Y,v}:X(k)\setminus Y\to \mathbb{R}$, well-defined up to $O(1)$, which gives a measure of the $v$-adic distance of a point to $Y$, being large when the point is close to $Y$. Moreover, one can associate a global height function $h_Y$, well-defined up to $O(1)$, which is a sum of appropriate local height functions.  If $Y=D$ is an effective (Cartier) divisor (which we will frequently identify with the associated closed subscheme), these height functions agree with the usual height functions associated to divisors.  Local height functions satisfy the following properties: if $Y$ and $Z$ are two closed subschemes of $X$, defined over $k$, and $v$ is a place of $k$, then up to $O(1)$,
\begin{align*}
\lambda_{Y\cap Z,v}&=\min\{ \lambda_{Y,v},\lambda_{Z,v}\},\\
\lambda_{Y+Z,v}&=\lambda_{Y,v}+\lambda_{Z,v},\\
\lambda_{Y,v}&\leq \lambda_{Z,v}, \text{ if }Y\subset Z.
\end{align*}
In particular, $\lambda_{Y,v}$ is bounded from below for all $P\in X(k)\setminus Y$.
  If $\phi:W\to X$ is a morphism of projective varieties with $\phi(W) \not \subset Y$, then up to $O(1)$,
\begin{equation*}
\lambda_{Y,v}(\phi(P))=\lambda_{\phi^*Y,v}(P), \quad \forall P\in W(k)\setminus \phi^*Y.
\end{equation*}
Here, $Y\cap Z$, $Y+Z$, $Y\subset Z$, and $\phi^*Y$ are defined in terms of the associated ideal sheaves (see \cite{silverman_87}).  In particular, we emphasize that if $Y$ corresponds to the ideal sheaf $\II_Y$, then $\phi^*Y$ is the closed subscheme corresponding to the inverse image ideal sheaf $\phi^{-1}\II_Y\cdot \mathcal{O}_W$.  Global height functions satisfy similar properties (except the first property above, which becomes $h_{Y\cap Z}\leq\min\{ h_{Y},h_{Z}\}+O(1)$).

\section{Proof of Theorem \ref{mthm} and Corollary \ref{repeated_cor}}\label{mthm_proof}

We now give the proof of the main theorem.

\begin{proof}[Proof of Theorem \ref{mthm}]
Let $\II_{i,v}$ be the coherent sheaf of ideals associated to $Y_{i,v}$, $\pi_{i,v}:\tilde{X}_{i,v}\to X$ be the blowing-up of $X$ along $Y_{i,v}$, and $E_{i,v}$ the associated divisor on $\tilde{X}_{i,v}$.  Let $\epsilon>0$ be a rational number, and for each $i$ and $v$ let $\epsilon_{i,v}\leq \epsilon_{Y_{i,v}}(A)$ be a positive rational number.  It follows from the definitions that $\pi_{i,v}^*A-\epsilon_{i,v}E_{i,v}$ is a nef $\mathbb{Q}$-divisor.  By Lemma \ref{amplelem}, for any sufficiently small positive rational number $\epsilon'$, depending on $\epsilon$, we have that $\epsilon\pi_{i,v}^*A-\epsilon'E_{i,v}$ is ${\mathbb{Q}}$-ample for all $i$ and $v$.  Therefore, 
\begin{align*}
(\pi_{i,v}^*A-\epsilon_{i,v}E_{i,v})+(\epsilon \pi_{i,v}^*A-\epsilon'E_{i,v})=(1+\epsilon)\pi_{i,v}^*A-(\epsilon_{i,v}+\epsilon')E_{i,v}
\end{align*}
is an ample $\mathbb{Q}$-divisor for all $i$ and $v$.  It follows from \cite[Exercise II.5.9(b)]{Hartshorne} that $\pi_{i,v}{}_*\O_{\tilde X_{i,v}}(-kE_{i,v})=\II_{i,v}^k$ for all sufficiently large integers $k$.  Let $N$ be a positive integer such that for all $i$ and for all $v\in S$, 
\begin{align}
\label{directimeq}
\pi_{i,v}{}_*\O_{\tilde X_{i,v}}(-N(\epsilon_{i,v}+\epsilon')E_{i,v})=\II_{i,v}^{N(\epsilon_{i,v}+\epsilon')}
\end{align}
and $N((1+\epsilon)\pi_{i,v}^*A-(\epsilon_{i,v}+\epsilon')E_{i,v})$ is a very ample divisor on $\tilde{X}_{i,v}$.  Fix $v\in S$.  We now construct divisors $F_{0,v},\ldots, F_{n,v}$ on $X$ such that
\begin{enumerate}
\item $N(1+\epsilon)A\sim F_{i,v}$, $i=0,\ldots, n$.\label{c1}
\item $\pi_{i,v}^*F_{i,v}\geq N(\epsilon_{i,v}+\epsilon')E_{i,v}$, $i=0,\ldots, n$.\label{c2}
\item The divisors $F_{0,v},\ldots, F_{n,v}$ are in general position on $X$.
\end{enumerate}

We define $F_{0,v},\ldots, F_{n,v}$ inductively as follows.  For some $j\in \{0,\ldots, n\}$, assume that $F_{0,v},\ldots, F_{j-1,v}$ have been defined so that \eqref{c1} and \eqref{c2} hold for $0\leq i\leq j-1$, and $F_{0,v},\ldots, F_{j-1,v},Y_{j,v},\ldots, Y_{n,v}$ are in general position on $X$ (for $j=0$ this reduces to the hypothesis that $Y_{0,v},\ldots, Y_{n,v}$ are in general position).  

Let $\tilde{F}^{(j)}_{i,v}=\pi_{j,v}^*F_{i,v}$, $i=0,\ldots, j-1$, and $\tilde{Y}^{(j)}_{i,v}=\pi_{j,v}^*Y_{i,v}$, $i=0,\ldots, n$. 
Since, in particular,  $F_{0,v},\ldots, F_{j-1,v},Y_{j+1,v},\ldots, Y_{n,v}$ are in general position on $X$, and $\pi_{j,v}$ is an isomorphism outside of $E_{j,v}$, the closed subschemes $\tilde{F}^{(j)}_{0,v},\ldots, \tilde{F}^{(j)}_{j-1,v},\tilde{Y}^{(j)}_{j+1,v},\ldots, \tilde{Y}^{(j)}_{n,v}$ are in general position on $\tilde{X}_{j,v}$ outside of $E_{j,v}$.  As $N((1+\epsilon)\pi_{j,v}^*A-(\epsilon_{j,v}+\epsilon')E_{j,v})$ is very ample, we can find a non-zero section $s\in H^0(\tilde X_{j,v}, \O_{\tilde X_{j,v}}(N((1+\epsilon)\pi_{j,v}^*A-(\epsilon_{j,v}+\epsilon')E_{j,v})))$ such that $\tilde{F}^{(j)}_{0,v},\ldots, \tilde{F}^{(j)}_{j-1,v},\dv(s), \tilde{Y}^{(j)}_{j+1,v},\ldots, \tilde{Y}^{(j)}_{n,v}$ are in general position on $\tilde{X}_{j,v}$ outside of $E_{j,v}$.  Let $\tilde{F}^{(j)}_{j,v}=\dv(s)+N(\epsilon_{j,v}+\epsilon')E_{j,v}$.  Then $\tilde{F}^{(j)}_{0,v},\ldots, \tilde{F}^{(j)}_{j-1,v}, \tilde{F}^{(j)}_{j,v}, \tilde{Y}^{(j)}_{j+1,v},\ldots, \tilde{Y}^{(j)}_{n,v}$ are in general position on $\tilde{X}_{j,v}$ outside of $E_{j,v}$.

We now claim that there is an effective divisor $F_{j,v}\sim N(1+\epsilon)A$ on $X$ such that $\pi_{j,v}^*F_{j,v}=\tilde{F}^{(j)}_{j,v}$.  For ease of notation, we temporarily set $\tilde{X}=\tilde X_{j,v}$, $\pi=\pi_{j,v}$, $\LL=\O_X(N(1+\epsilon)A)$, $\II=\II_{j,v}^{N(\epsilon_{j,v}+\epsilon')}$, and $\tilde{\II}=\O_{\tilde X_{j,v}}(-N(\epsilon_{j,v}+\epsilon')E_{j,v}))$.  Then $\tilde{\II}=\pi^{-1}\II\cdot \O_{\tilde X}$ and, by our choice of $N$, $\pi_*\tilde{\II}=\II$.

We have a map of sheaves on $\tilde{X}$,
\begin{align*}
\pi^*\II\to \tilde{\II}\to \O_{\tilde{X}}=\pi^*\O_X,
\end{align*}
where the composite map is induced by the ideal sheaf map $\II\to \O_X$ \cite[Caution II.7.12.2]{Hartshorne}.  Tensoring with $\pi^*\LL$, we obtain a commutative diagram
\begin{equation*}
\begin{tikzcd}
\pi^*(\LL\otimes \II) \arrow{r}{} \arrow[swap]{d}{}& \pi^*\LL\arrow{d}{}\\
\pi^*\LL\otimes \tilde{\II} \arrow{r}{} & \pi^*\LL
\end{tikzcd}
\end{equation*}
where the top map is induced by $\LL\otimes \II\to \LL$ and the vertical map on the right is the identity.  Using that $\pi_*$ is right adjoint to $\pi^*$ yields a diagram
\begin{equation*}
\begin{tikzcd}
\LL\otimes \II \arrow{r}{} \arrow[swap]{d}{}& \LL\arrow{d}{}\\
\pi_*(\pi^*\LL\otimes \tilde{\II}) \arrow{r}{} & \pi_*\pi^*\LL
\end{tikzcd}
\end{equation*}
which is commutative by the naturality of the adjunction maps, and where the vertical map on the left is the isomorphism of the projection formula \cite[Exercise II.5.1]{Hartshorne}.  Taking global sections, using the definition of the direct image sheaf, and reverting to the fuller notation, we obtain a commutative diagram
\begin{equation*}
\begin{tikzcd}
H^0(X,\O_X(N(1+\epsilon)A)\otimes \II_{j,v}^{N(\epsilon_{j,v}+\epsilon')}) \arrow{r}{} \arrow[swap]{d}{\cong}& H^0(X,\O_X(N(1+\epsilon)A))\arrow{d}{\pi_{j,v}^*}\\
H^0(\tilde X_{j,v}, \O_{\tilde X_{j,v}}(N((1+\epsilon)\pi_{j,v}^*A-(\epsilon_{j,v}+\epsilon')E_{j,v}))) \arrow{r}{}  & H^0(\tilde X_{j,v}, \O_{\tilde X_{j,v}}(N((1+\epsilon)\pi_{j,v}^*A)))
\end{tikzcd}
\end{equation*}
where the left-hand isomorphism comes from the projection formula and the horizontal maps are induced by the maps of ideal sheaves 
\begin{align*}
\II_{j,v}^{N(\epsilon_{j,v}+\epsilon')}&\to \O_X,\\
\O_{\tilde X_{j,v}}(-N(\epsilon_{j,v}+\epsilon')E_{j,v})&\to \O_{\tilde X_{j,v}}.
\end{align*}

Then the claim (in the language of sections) follows from the commutative diagram and the definition of $\tilde{F}^{(j)}_{j,v}$: there is an effective divisor $F_{j,v}\sim N(1+\epsilon)A$ on $X$ such that $\pi_{j,v}^*F_{j,v}=\tilde{F}^{(j)}_{j,v}$.

It is immediate that $F_{j,v}$ satisfies the conditions \eqref{c1} and \eqref{c2} above for $i=j$.   To complete the inductive definition, it remains to show that $F_{0,v}, \ldots, F_{j,v}, Y_{j+1,v},\ldots, Y_{n,v}$ are in general position.  Since $\tilde{F}^{(j)}_{0,v},\ldots, \tilde{F}^{(j)}_{j-1,v},\tilde{F}^{(j)}_{j,v},\tilde{Y}^{(j)}_{j+1,v},\ldots, \tilde{Y}^{(j)}_{n,v}$ are in general position on $\tilde{X}_{j,v}$ outside of $E_{j,v}$, and $\pi_{j,v}$ is an isomorphism above the complement of $Y_{j,v}$, it is clear that $F_{0,v}, \ldots, F_{j,v}, Y_{j+1,v},\ldots, Y_{n,v}$ are in general position outside of $Y_{j,v}$.  The full statement now follows from combining this with the fact that $Y_{j,v}$ is in general position with $F_{0,v},\ldots, F_{j-1,v},Y_{j+1,v},\ldots, Y_{n,v}$.  Thus, we obtain divisors $F_{0,v},\ldots, F_{n,v}$ with the required properties.

\par
We may now apply Theorem~\ref{EF} to the linearly equivalent divisors $F_{i,v}$, $i=0,\ldots,n$, $v\in S$, and $N(1+\epsilon)A$.  We obtain
\begin{align*}
\sum_{v\in S}\sum_{i=0}^n \lambda_{F_{i,v},v}(P)< (n+1+\epsilon)h_{N(1+\epsilon)A}(P)
\end{align*}
for all $P\in X(k)\setminus Z$ for some proper Zariski-closed subset $Z$ of $X$ containing the supports of all $F_{i,v}$, $v\in S$, $i=0,\ldots, n$.  By functoriality, additivity,  and the fact that local height functions attached to effective divisors are bounded from below outside their support, we have
\begin{align*}
\lambda_{F_{i,v},v}(\pi_{i,v}(P))&=\lambda_{\pi_{i,v}^*F_{i,v},v}(P)+O(1)\\
&\geq N(\epsilon_{i,v}+\epsilon')\lambda_{E_{i,v},v}(P)+O(1)\\
&= N(\epsilon_{i,v}+\epsilon')\lambda_{Y_{i,v},v}(\pi_{i,v}(P))+O(1)
\end{align*}
for all $P\in \tilde{X}_{i,v}(k)$ outside the support of $\pi_{i,v}^*F_{i,v}$.

Then there exists a Zariski closed subset $Z\subset X$ such that
\begin{equation*}
\sum_{v\in S}\sum_{i=0}^n (\epsilon_{i,v}+\epsilon')\lambda_{Y_{i,v},v}(P)< (1+\epsilon)(n+1+\epsilon)h_A(P)+O(1)
\end{equation*}
for all $P \in X(k)\setminus Z$.  Since $\epsilon, \epsilon'$, and $\epsilon_{Y_{i,v}}(A)-\epsilon_{i,v}$, for all $i$ and $v$, may be chosen (simultaneously) arbitrarily small, the result follows.
\end{proof}

We end this section by giving the short proof of Corollary \ref{repeated_cor}.

\begin{proof}[Proof of Corollary \ref{repeated_cor}]
Let $v\in S$.  If $\codim Y_v=r$ and $Y_{i,v}=Y_v$, $i=0,\ldots, r-1$, then the closed subschemes $Y_{0,v},\ldots, Y_{r-1,v}$ are in general position.  The result is now immediate from Theorem \ref{mthm}, after choosing the remaining closed subschemes $Y_{i,v}$ arbitrarily (so that the general position assumption is maintained) and using that local height functions associated to closed subschemes are bounded from below outside their support.
\end{proof}

\section{Comparing $\beta_{A,Y}$ and $\epsilon_Y(A)$}
\label{beta_ineq}
In this final section, we combine the method of proof of our main theorem with work of Autissier \cite{autissier} to establish inequality \eqref{Seshineq}.  We begin by recalling the setup of Autissier's filtrations.  Let $\LL$ be a line bundle on a non-singular projective variety $X$ of dimension $n$, and $D_1,\ldots, D_r$ ample effective divisors on $X$.  We assume that $h^0(X, \LL)\geq 1$, $D_1,\ldots, D_r$ are in general position on $X$, and $\cap_{i=1}^rD_i$ is non-empty.  Under our assumptions, the general position condition is equivalent to Autissier's assumption in \cite{autissier} that $D_1,\ldots, D_r$ intersect properly \cite[Remarque 2.3]{autissier}.

Let $\mathbb{R}^+=[0,\infty)$ and let
\begin{align*}
\Delta=\{\bt=(t_1,\ldots, t_r)\in (\mathbb{R}^+)^r\mid t_1+\cdots +t_r=1\}.
\end{align*}

For each $\bt\in \Delta$ and $x\in \mathbb{R}^+$ define
\begin{align}
\label{Ntx}
N(\bt,x)&=\{\mathbf{b}\in \mathbb{N}^r\mid t_1b_1+\cdots +t_rb_r\geq x\},\\
\II(\bt,x)&=\sum_{\mathbf{b}\in N(\bt, x)}\O_X\left(-\sum_{i=1}^rb_iD_i\right),\notag\\
\FF(\bt)_x&=H^0(X,\II(\bt,x)\otimes \LL).\notag
\end{align}

For a section $s\in H^0(X,\LL)\setminus\{0\}$, let $\mu_{\bt}(s)=\sup\{y\in \mathbb{R}^+\mid s\in \FF(\bt)_y\}$.  Let
\begin{align*}
\FF(\bt)=\frac{1}{h^0(X,\LL)}\int_0^\infty (\dim \FF(\bt)_x)dx.
\end{align*}

If $\BB=\{s_1,\ldots, s_l\}$ is a basis of $H^0(X,\LL)$ adapted to the filtration $(\FF(\bt)_x)_{x\in\mathbb{R}^+}$,  then \cite[Remarque 3.5]{autissier}
\begin{align*}
\FF(\bt)=\frac{1}{l}\sum_{k=1}^l\mu_{\bt}(s_k).
\end{align*}

We use the following theorem \cite[Th\'eor\`eme 3.6]{autissier}.

\begin{theorem}
The function $\FF$ is concave on $\Delta$.  In particular, for $\bt\in \Delta$,
\begin{align}
\label{Auteq}
\FF(\bt)\geq \min_{i=1,\ldots,r} \frac{1}{h^0(X,\LL)}\sum_{m\geq 1}h^0(X,\LL(-mD_i)).
\end{align}
\end{theorem}

Suppose that $D\sim dA$ for some positive integer $d$ and ample divisor $A$ on $X$.  We will use the formula
\begin{align}
\label{ADeq}
\lim_{N\to\infty}\frac{1}{Nh^0(X, NA)}\sum_{m\geq 1}h^0(X,NA-mD)=\frac{1}{d(n+1)},
\end{align}
which is an easy consequence of the asymptotic Riemann-Roch formula 
\begin{align*}
h^0(X,mA)=\frac{A^n}{n!}m^n+O(m^{n-1}).
\end{align*}

We are now in a position to prove inequality \eqref{Seshineq}, restated here as
\begin{theorem}\label{shineq_thm}
Let $X$ be a non-singular projective variety of dimension $n$ and let $Y$ be a closed subscheme of $X$ of codimension $r$.  Let $A$ be an ample divisor on $X$.  Then
\begin{align*}
\beta_{A,Y}\geq \frac{r}{n+1}\epsilon_Y(A).
\end{align*}
\end{theorem}

\begin{proof} Let $\pi:\tilde X\to X$ be the blowing-up of $X$ along $Y$, and let $E$ be the associated divisor as above.
Let $\epsilon$ and $\epsilon_{Y}' \leq \epsilon_Y(A)$ be positive rational numbers.
As in the proof of Theorem \ref{mthm}, for all sufficiently small $\epsilon'$, and for any sufficiently large and divisible positive integer $M$ (depending on $\epsilon$ and $\epsilon'$), we construct effective divisors $F_1,\ldots, F_r$ on $X$ such that
\begin{enumerate}
\item $M(1+\epsilon)A\sim F_i$, $i=1,\ldots, r$.
\item $\pi^*F_i\geq M(\epsilon_{Y}'+\epsilon')E$, $i=1,\ldots, r$.
\item The divisors $F_1,\ldots, F_r$ are in general position on $X$.
\end{enumerate}

Let $N$ be a positive integer and $\LL=\O_X(NA)$.  We take $\bt_0=(\frac{1}{r},\ldots, \frac{1}{r})$ and $D_i=F_i, i=1,\ldots, r$.  Let $\BB=\{s_1,\ldots, s_l\}$ be a basis of $H^0(X,\LL)$ adapted to the filtration $(\FF(\bt_0)_x)_{x\in\mathbb{R}^+}$, where $l=h^0(X,\LL)$.  Then
\begin{align*}
\FF(\bt_0)=\frac{1}{l}\sum_{k=1}^l\mu_{\bt_0}(s_k).
\end{align*}

From the construction of $F_i$ in the proof of Theorem \ref{mthm}, we have
\begin{align*}
\II_Y^{M(\epsilon_{Y}'+\epsilon')}\supset \O_X\left(-F_i\right), \quad i=1,\ldots, r.
\end{align*}
Therefore, if $s\in \FF(\bt_0)_y$ then
\begin{align*}
s\in H^0\left(X, \LL\otimes\II_Y^{\left\lceil ry\right\rceil M(\epsilon_{Y}'+\epsilon')}\right),
\end{align*}
where the use of $\lceil ry\rceil$ is justified by the observation that if 
\begin{align*}
t_1b_1+\cdots +t_rb_r=\frac{1}{r}(b_1+\cdots +b_r)\geq x 
\end{align*}
in \eqref{Ntx}, then $b_1+\cdots +b_r\geq \lceil rx\rceil$ as $b_1,\ldots, b_r\in\mathbb{N}$ (for the same reason, $r\mu_{\bt_0}(s)$ is an integer for any $s\in H^0(X,\LL)$).

From the projection formula (and an appropriate choice of $M$),
\begin{align*}
H^0(X, \LL\otimes\II_Y^{\left\lceil ry\right\rceil M(\epsilon_{Y}'+\epsilon')})\cong H^0(\tilde{X},\O_{\tilde{X}}(N\pi^*A-\left\lceil ry\right\rceil M(\epsilon_{Y}'+\epsilon')E)).
\end{align*}

It follows that a given section $s_k$ in $\BB$ corresponds to an element of $H^0(\tilde{X},\O_{\tilde{X}}(N\pi^*A-mE))$ for at least $m=1,\ldots, rM(\epsilon_{Y}'+\epsilon')\mu_{\bt_0}(s_k)$, and therefore

\begin{align*}
\frac{\sum_{m=1}^\infty h^0(\tilde{X},N\pi^*A-mE)}{h^0(X,NA)}\geq rM(\epsilon_{Y}'+\epsilon')\frac{1}{l}\sum_{k=1}^l\mu_{\bt_0}(s_k)=rM(\epsilon_{Y}'+\epsilon')\FF(\bt_0).
\end{align*}

From \eqref{Auteq} we find that 
\begin{align*}
\FF(\bt_0)\geq \min_{i=1,\ldots,r} \frac{1}{l}\sum_{m\geq 1}h^0(X,\LL(-mF_i)).
\end{align*}
By \eqref{ADeq}, for any $\epsilon''>0$ and sufficiently large $N$, we have
\begin{align*}
\frac{1}{Nh^0(X,\LL)}\sum_{m\geq 1}h^0(X,\LL(-mF_i))\geq \frac{1}{M(n+1)(1+\epsilon)}-\epsilon'', \quad i=1,\ldots, r.
\end{align*}
Then for sufficiently large $N$,
\begin{align*}
\frac{\sum_{m=1}^\infty h^0(\tilde{X},N\pi^*A-mE)}{Nh^0(X,NA)}\geq \frac{r(\epsilon_{Y}'+\epsilon')}{(n+1)(1+\epsilon)}-rM(\epsilon_{Y}'+\epsilon')\epsilon''.
\end{align*}

Since we may choose $\epsilon$, $\epsilon'$, and $\epsilon_Y(A)-\epsilon_Y'$ arbitrarily small, and then choose $\epsilon''$ so that $rM(\epsilon_{Y}'+\epsilon')\epsilon''$ is arbitrarily small, we find that

\begin{align*}
\beta_{A,Y}=\lim_{N\to\infty} \frac{\sum_{m=1}^\infty h^0(\tilde{X},N\pi^*A-mE)}{Nh^0(X,NA)}\geq \frac{r}{n+1}\epsilon_Y(A)
\end{align*}
as desired.
\end{proof}

\end{document}